\documentclass[12pt,reqno]{amsart}

\usepackage{amsmath,amsfonts,amsthm,amssymb,amsxtra}
\usepackage{bbm} % to get \1
\usepackage{enumerate}   
\usepackage[parfill]{parskip}    % Activate to begin paragraphs with an empty line rather than an indent
\usepackage[colorlinks=true,hypertexnames=false]{hyperref}
\hypersetup{urlcolor=blue, citecolor=blue}

\newtheorem{theorem}{Theorem}[section]
\newtheorem{proposition}[theorem]{Proposition}
\newtheorem{lemma}[theorem]{Lemma}
\newtheorem{corollary}[theorem]{Corollary}

\theoremstyle{definition}

\theoremstyle{remark}

\newtheorem{remarks}[theorem]{Remarks}

\numberwithin{equation}{section}

\renewcommand{\epsilon}{\varepsilon}

\newcommand{\N}{\mathbb{N}}

\renewcommand{\phi}{\varphi}
\newcommand{\R}{\mathbb{R}}

\newcommand{\Sph}{\mathbb{S}}

\newcommand{\eps}{{\epsilon}}

\begin{document}

\title[Stability of reverse Sobolev inequalities]{Stability inequalities with explicit constants for a family of reverse Sobolev inequalities on the sphere}

\author{Tobias K\"onig}
\address[Tobias K\"onig]{Institut für Mathematik, 
Goethe-Universität Frankfurt, 
Robert-Mayer-Str. 10, 60325 Frankfurt am Main, Germany}
\email{koenig@mathematik.uni-frankfurt.de}

\thanks{\copyright\, 2025 by the author. This paper may be reproduced, in its entirety, for non-commercial purposes. The author warmly thanks Jean Dolbeault and Maria J. Esteban for many stimulating discussions.   
}

\maketitle
\thispagestyle{empty}

%%%%%%%%%%%%%%%%%%%%%%%%%%%%%%%%%%%%%%%%%%%%%%%%%%%%%%%%%%%%%%%%%%%%%%%%%%
%%%%%%%%%%%%%%%%%%%%%%%%%%%%%%%%%%%%%%%%%%%%%%%%%%%%%%%%%%%%%%%%%%%%%%%%%%

\begin{abstract}
We prove a stability inequality associated to the reverse Sobolev inequality on the sphere $\mathbb S^n$ from \cite{FrKoTa2022}, for the full admissible parameter range $s - \frac{n}{2} \in (0,1) \cup (1,2)$. To implement the classical proof of Bianchi and Egnell \cite{BiEn}, we overcome the main difficulty that the underlying operator $A_{2s}$ is not positive definite. 

As a consequence of our analysis and recent results from \cite{Gong2025}, the case $s - \frac{n}{2} \in (1,2)$ remarkably constitutes the first example of a Sobolev-type stability inequality (i) whose best constant is explicit and (ii) which does not admit an optimizer. 
\end{abstract}

\section{Introduction and main results}

Throughout, we fix a space dimension $n \geq 1$, which for simplicity will not be reflected in the notation for constants and other objects.  

\subsection{The Sobolev inequality on $\mathbb S^n$}
For $s \in (0, \frac{n}{2})$, the standard fractional Sobolev inequality on the sphere $\mathbb S^n$, due to Beckner \cite{Beckner1993}, states that there is $\mathcal S_s >0$ such that 
\begin{equation}
\label{sobolev standard}
 \|A_{2s}^\frac{1}{2} u\|_2^2 \geq \mathcal S_s \mathcal \|u\|_\frac{2n}{n-2s}^2 \qquad \text{ for all } u \in H^s(\mathbb S^n). 
\end{equation}
Here, the operator $A_{2s}$ is given through its action on spherical harmonics of degree $\ell \geq 0$ by
\begin{equation}
\label{alpha ell definition}
A_{2s} Y_\ell = \alpha_{2s}(\ell) Y_\ell, \quad \text{ where } \quad 
\alpha_{2s}(\ell) = \frac{\Gamma(\ell+\frac n2 + s)}{\Gamma(\ell+\frac n2 - s)}.
\end{equation}
Inequality \eqref{sobolev standard} is a conformally equivalent version, through stereographic projection, of the fractional Sobolev inequality $\|(-\Delta)^\frac{s}{2} v\|_{L^2(\R^n)}^2 \geq \mathcal S_s \|v\|_{L^\frac{2n}{n-2s}(\R^n)}^2$ on $\R^n$.  

The best constant in \eqref{sobolev standard} is explicitly given by
\begin{equation}
\label{Ss definition}
\mathcal S_s = (4\pi)^s \ \frac{\Gamma(\frac{n+2s}{2})}{\Gamma(\frac{n-2s}{2})} \left( \frac{\Gamma(\frac n2)}{\Gamma(n)} \right)^{2s/n}
\end{equation}
and is attained if and only if $u \in \widetilde{\mathcal M}$, where 
\begin{equation}
\label{M definition}
\widetilde{\mathcal M} := \left\{ \omega \mapsto c (1 - \zeta \cdot \omega)^{-\frac{n-2s}{2}} \, : \, c \in \R \setminus \{0\}, \, \zeta \in \R^{n+1}, \, |\zeta| < 1 \right\}. 
\end{equation}

\subsection{The reverse Sobolev inequality}

For larger values of $s$, that is, $s > \frac{n}{2}$, the operator $A_{2s}$ can still be defined by \eqref{alpha ell definition} (with the convention that $\alpha_{2s}(\ell) = 0$ if $\ell + \frac{n}{2} - s \in \mathbb Z_{\leq 0}$), but it  fails to be positive definite. Moreover, the exponent $\frac{2n}{n-2s}$ appearing in \eqref{sobolev standard} becomes negative for $s > \frac{n}{2}$.

Somewhat surprisingly, an inequality like \eqref{sobolev standard} nevertheless continues to be true for certain values of $s > \frac{n}{2}$.  Indeed, in \cite{FrKoTa2022} (see also \cite{Ha, HaYa}) the following inequality has been obtained: 
\begin{equation}
\label{sobolev reverse}
 a_{2s}[u] \geq \mathcal S_s \|u\|_\frac{2n}{n-2s}^2  \qquad \text{ for all } 0 < u \in H^s(\mathbb S^n), \quad \text{ when } s - \frac{n}{2} \in (0,1) \cup (1,2). 
\end{equation} 
Here and in the following, for any $q \in \R \setminus \{0\}$ and $v > 0$, we denote $\|v\|_q := \left( \int_{\Sph^n} v^q \, d \omega \right)^\frac{1}{q}$. The quadratic form $a_{2s}[u]$ that appears in \eqref{sobolev reverse} is defined as 
\begin{equation}
\label{a2s definition}
a_{2s}[u] := \sum_{\ell = 0}^\infty \alpha_{2s}(\ell) \|P_\ell u\|_2^2, 
\end{equation}
where $P_\ell$ is the $L^2$-orthogonal projection on the space of spherical harmonics of degree $\ell$. Since $a_{2s}[u]$ coincides with $\|A_{2s}^{1/2}\|_2^2$ when $s \in (0, \frac{n}{2})$, inequality \eqref{sobolev reverse} is the natural extension of \eqref{sobolev standard} to $s > \frac{n}{2}$.

The best constant $\mathcal S_s$ in \eqref{sobolev reverse} is still given by \eqref{Ss definition} and is attained if and only if $u \in \mathcal M$, where\footnote{There is only one difference between $\mathcal M$ defined in \eqref{sobolev optimizers} for $s - \frac{n}{2} \in (0,1) \cup (1,2)$ and $\widetilde{\mathcal M}$ defined in \eqref{M definition} for $s < \frac{n}{2}$: for $\mathcal M$ only $c > 0$ is allowed in view of the restriction $u > 0$ in \eqref{sobolev reverse}, while for $\widetilde{ \mathcal M}$ the prefactor $c$ can also take negative values.}
\begin{equation}
\label{sobolev optimizers}
\mathcal M := \left\{ \omega \mapsto  c (1 - \zeta \cdot \omega)^{\frac{2s-n}{2}} \, : \,  c >0, \, \zeta \in \R^{n+1},  \,  |\zeta| < 1 \right\}. 
\end{equation}

However, one sees from \eqref{Ss definition} that the constant $\mathcal S_s$ can be negative, namely when $ s - \frac{n}{2} \in (0,1)$.  This sign change justifies  calling \eqref{sobolev reverse} a \emph{reverse Sobolev inequality}.

We add for completeness that for $s - \frac{n}{2} \in (2, \infty) \setminus \N$,  inequality \eqref{sobolev reverse} fails, at least when $n \geq 2$ \cite[Theorem 2]{FrKoTa2022}. (It still holds trivially when $s - \frac{n}{2} \in \N$, with best constant $\mathcal S_s = 0$. We do not explore the case $s - \frac{n}{2} \in \N$ further in this paper.)

Despite not being directly relevant to this paper, it is instructive to compare, as an aside, the reverse Sobolev inequality \eqref{sobolev reverse} with a similar reverse-type inequality, namely the \emph{reverse HLS inequality} of Dou and Zhu \cite{DoZh}. For every $s > \frac{n}{2}$, this inequality states that there is $\mathcal H_s > 0$ such that
\begin{equation}
	\label{HLS reverse}
	\iint_{\Sph^n\times\Sph^n} v(\omega) |\omega-\omega'|^{2s-n} v(\omega') \,d\omega\,d\omega' \geq \mathcal H_s \|v\|_{\frac{2n}{n+2s}}^2\,, \quad \text{ for all } 0 < v\in L^\frac{2n}{n+2s}(\Sph^n).
\end{equation}
(Notice the inequality sign $\geq$ here, as compared to $\leq$ in the standard HLS inequality \cite[Theorem 4.3]{Lieb2001}.)
Similarly to \eqref{Ss definition} and \eqref{M definition}, the best constant $\mathcal H_s$ is explicit and the optimizers are classified.

It is interesting to note that, differently from \eqref{sobolev reverse}, inequality \eqref{HLS reverse} holds for any $s > \frac{n}{2}$. Despite the fact that the kernel $|\omega - \omega'|^{2s-n}$ is not positive definite for $s > \frac{n}{2}$, in the recent preprint \cite{Gong2025}, the authors succeed in deriving \eqref{sobolev reverse} from \eqref{HLS reverse} through a  duality argument when $s - \frac{n}{2} \in (0,1)$.

\subsection{The stability of the reverse Sobolev inequality}

For the standard Sobolev inequality \eqref{sobolev standard}, a classical result by Bianchi and Egnell \cite{BiEn} (and its generalization in \cite{Chen2013}) states that for every $s \in (0, \frac{n}{2})$ there is $c_{BE}(s) >0$ such that for all $u \in H^s(\mathbb S^n)$, one has 
\begin{equation}
\label{BE classical}
\|A_{2s}^{\frac{1}{2}} u\|_2^2 - \mathcal S_s \|u\|_\frac{2n}{n-2s}^2 \geq c_{BE}(s) d(u). 
\end{equation} 
Here, 
\begin{equation}
\label{d classical}
d(u) := \inf_{h \in \widetilde{\mathcal M}} \|A_{2s}^\frac{1}{2} (u-h)\|_2^2
\end{equation}  
is the distance of $u$ to the manifold $\widetilde{\mathcal M}$ of optimizers of \eqref{sobolev standard}. 

Inequality \eqref{BE classical} implies that the only way the deficit $\|A_{2s}^{\frac{1}{2}} u\|_2^2  - \mathcal S_s \|u\|_\frac{2n}{n-2s}^2$ can become small is by $u$ being close to $\widetilde{\mathcal M}$. What is more, it gives an explicit control on this closeness in terms of the deficit. For these reasons, \eqref{BE classical} is usually referred to as a \emph{quantitative stability inequality}. 

As remarked in \cite{FrKoTa2022}, the extension of \eqref{BE classical} to the setting of the reverse Sobolev inequality \eqref{sobolev reverse} is not straightforward. Indeed, since $A_{2s}$ fails to be positive definite when $s > \frac{n}{2}$, the quantity $d(u)$ defined in \eqref{d classical} may be negative or even $-\infty$, in which case inequality  \eqref{BE classical} becomes meaningless. 

In the recent preprint \cite{Gong2025}, the authors partially overcome this difficulty. Applying a center of mass condition already used in a similar context by Hang \cite{Ha}, they prove that for $s - \frac{n}{2} \in (1,2)$ and every $0 < u \in H^s(\mathbb S^n) \setminus \mathcal M$ there exists $h \in \mathcal M$ such that $a_{2s}[u-h] > 0$ and 
\begin{equation}
\label{BE GYZ}
a_{2s}[u] - \mathcal S_s \|u\|_\frac{2n}{n-2s}^2 \geq a_{2s}[u-h]. 
\end{equation} 
Remarkably, the proof of \eqref{BE GYZ} in \cite{Gong2025} does not follow the classical two-step strategy from \cite{BiEn}, where a compactness argument permits to reduce the proof of \eqref{BE classical} to a local analysis near $\mathcal M$. Instead, the argument in \cite{Gong2025} is much more direct and, besides the center of mass condition, only uses a simple reverse Hölder inequality to conclude \eqref{BE GYZ}. As a consequence,  the non-explicit constant $c_{BE}(s)$ from \eqref{BE quotient def Q} can be replaced by $1$ in \eqref{BE GYZ}. 

For $s - \frac{n}{2} \in (0,1)$, however, the preprint \cite{Gong2025} does not prove any stability inequality. Moreover, the argument in \cite{Gong2025} does not make clear in what sense \eqref{BE GYZ} is sharp. 

Our main goal in this paper is to close these gaps and complete the stability analysis for the full range $s - \frac{n}{2} \in (0,1) \cup (1,2)$ of the reverse Sobolev inequality \eqref{sobolev reverse} by making Bianchi's and Egnell's strategy amenable to the reverse case. In order to do so, a key step for us is to replace $d(u)$ by the following modified distance: For $0 < u \in H^s(\mathbb S^n)$, we set 
\begin{equation}
    \label{d(u) definition}
    \mathsf d(u) := \inf \left\{ a_{2s}[\rho] \, : \, u = h + \rho \text{ for } h \in \mathcal M \text{ and } \rho \in (T_h \mathcal M)^\perp  \right\}. 
\end{equation} 
Here $T_h \mathcal  M$ denotes the tangent space of the manifold $\mathcal M$ in the point $h \in \mathcal M$. The orthogonal complement is intended with respect to $a_{2s}$, that is,
\[ (T_h \mathcal M)^\perp := \{ \rho \in H^s(\mathbb S^n) \, : \, a_{2s}[\rho, \varphi] = 0 \, \text{ for all } \, \varphi \in T_h \mathcal M \}. \] 
A few remarks concerning Definition \eqref{d(u) definition} are in order. 

\begin{remarks}
\begin{enumerate}[(i)]
\item We show in Proposition \ref{proposition balance condition} that $\mathsf d(u) > 0$ whenever $u \in H^s(\mathbb S^n) \setminus \mathcal M$. The center of mass condition used in \cite{Gong2025} for a similar purpose also enters here, namely to show that the set on the right side of \eqref{d(u) definition} is not empty.

\item For the standard range $s \in (0, \frac{n}{2})$ our definition of $\mathsf d(u)$ {coincides} with the standard quantity $d(u)$ (defined in \eqref{d classical}) in the Bianchi-Egnell inequality \eqref{BE classical}. This is because any minimizer $h \in \mathcal M$ of \eqref{d classical} satisfies $u-h \in (T_h \mathcal M)^\perp$ by minimality. 

In some sense, our approach for $s - \frac n2 \in (0, 2)$ can be viewed as detecting \emph{critical points} of $\mathcal M \ni h \mapsto a_{2s}[u-h]$ instead of minimizers (which as discussed above need not exist). 

\item In general, the set on the right side of \eqref{d(u) definition} can consist of more than one element. Also the infimum in \eqref{d(u) definition} can be attained by more than one configuration. Both of this can be checked to be the case, e.g., for a two-bubble configuration of the form 
\[ u_\beta(\omega) = (1 + \beta \omega_{n+1})^\frac{n-2s}{2} +  (1 - \beta \omega_{n+1})^\frac{n-2s}{2}, \]
with $\beta \in (0,1)$ close to $1$. 
\end{enumerate}
\end{remarks}

\subsection{Main results}

With $\mathsf d(u)$ defined as in \eqref{d(u) definition}, we define the stability quotient functional as 
\begin{equation}
\label{BE quotient def Q}
    \mathcal E(u) := \frac{a_{2s}[u] - \mathcal S_s \|u\|_\frac{2n}{n-2s}^2 }{\mathsf d(u)} , \qquad 0 < u \in H^s(\mathbb S^n) \setminus \mathcal M.
\end{equation} 
and the associated best stability constant as 
\begin{equation}
\label{c BE definition}
c_{BE}(s) := \inf_{0 < u \in H^s(\mathbb S^n) \setminus \mathcal M} \mathcal E(u). 
\end{equation}

Now we can state our main results on the quantitative stability of the reverse Sobolev inequality. 

\begin{theorem}
\label{theorem Bianchi Egnell new}
Let $\mathcal E(u)$ be defined by \eqref{BE quotient def Q} and let $s - \frac{n}{2} \in (0,1) \cup (1,2)$. Then $c_{BE}(s) > 0$. Moreover, the following holds. 
\begin{enumerate}[(a)]
\item If $s - \frac{n}{2} \in (0,1)$, then $c_{BE}(s) \leq \frac{4s}{n+2s+2}$. 
\item If $s - \frac{n}{2} \in (1,2)$, then $c_{BE}(s) = 1$ and $c_{BE}(s)$ is not attained. 
\end{enumerate}
%\texttt{What about $s = 1$? }
\end{theorem}

Part (b) of the above theorem is particularly remarkable for two reasons:
\begin{itemize}
\item It is, to the best of our knowledge, the first instance of a stability inequality associated to a Sobolev-type inequality  for which the sharp constant $c_{BE}(s)$ can be determined! This is in striking contrast to the Bianchi-Egnell stability inequality \eqref{BE classical} for the standard range $s \in (0, \frac{n}{2})$: despite its more than 30 years of history, even non-sharp lower bounds have only been obtained recently by novel methods \cite{Dolbeault2025}.
\item  It is, moreover, the first example of such an inequality which does not possess a minimizer.  (For $n = 1$ and $s < \frac{1}{2}$, non-existence of a minimizer for \eqref{BE classical} is conjectured in \cite{Koenig2024}. For the Caffarelli-Kohn-Nirenberg inequality, non-existence of a stability minimizer for certain parameter values is plausible by the results in  \cite{Wei2024, Deng2023}.) 
\end{itemize}

Since Theorem \ref{theorem Bianchi Egnell new} leaves open whether or not $c_{BE}(s)$ is attained for $s - \frac{n}{2} \in (0,1)$, we aim to investigate this question further. Indeed, the standard stability inequality \eqref{BE classical} has been shown to admit a minimizer in the recent paper \cite{Koenig2022-stab} for all $n \geq 2$ and $0 < s < \frac{n}{2}$.

A crucial step in the existence proof of a stability minimizer from \cite{Koenig2022-stab} is the strict inequality with respect to the 'local constant' $\frac{4s}{n+2s+2}$ (shown in \cite{Koenig2023} for $0 < s < \frac{n}{2}$). This strict inequality turns out to also hold in the reverse case. 

\begin{theorem}
\label{theorem strict inequality}
Let $s - \frac{n}{2} \in (0,1)$ and suppose that $n \geq 2$. Then $c_{BE}(s) < \frac{4s}{n +2s + 2}$. 
\end{theorem}

Arguing along the lines of \cite{Koenig2022-stab}, using Theorem \ref{theorem strict inequality} we then recover the following partial result.  

\begin{proposition}
\label{theorem minimizer}
Let $s - \frac{n}{2} \in (0,1)$ and suppose that $n \geq 2$. Let $(u_k)$ be a minimizing sequence for $c_{BE}(s)$. Then there exist $c_k \in (0,\infty)$, conformal maps $\Phi_k$ and $v \in H^s(\mathbb S^n) \setminus \mathcal M$ with $v \geq 0$ such that $c_k u_{\Phi_k} =: v_k \to v$ weakly in $H^s(\mathbb S^n)$ and strongly in $L^\infty(\mathbb S^n)$. 
\end{proposition}

The fact that we are not able to strengthen the conclusion of Proposition \ref{theorem minimizer} is linked to two difficulties. Firstly, while the deficit $a_{2s}[u] - \mathcal S_s\|u\|_\frac{2n}{n-2s}^2$ also makes sense for $u \geq 0$ (see, e.g., \cite[Remark 1.1]{Gong2025}), it is not clear to us how to meaningfully extend our definition of $\mathsf d(u)$ to general nonnegative functions $u \geq 0$. More specifically, we cannot prove that Lemma \ref{lemma balance condition} holds under the weaker assumption $u \geq 0$. Secondly, even for $u > 0$ it is not clear to us how the distance $\mathsf d(u_k)$ behaves asymptotically under a splitting $u_k = f + g_k$ with $g_k \rightharpoonup 0$ weakly in $H^s(\mathbb S^n)$ (compare \cite[Lemma 4.2]{Koenig2022-stab}). Such an information seems needed to conclude strong $H^s$-convergence of $v_k$ to $v$, similarly to \cite[Proof of Theorem 1.2]{Koenig2022-stab}.  In particular, Proposition \ref{theorem minimizer} \emph{does not imply that $c_{BE}(s)$ is attained for $s - \frac{n}{2} \in (0,1)$}.

We find it interesting that the possibility that $u$ vanishes somewhere yields another plausible mechanism for non-compactness of minimizing sequences for a Bianchi-Egnell-type stability inequality which is qualitatively different from those that had to be addressed in \cite{Koenig2023, Koenig2022-stab, Koenig2024}, namely convergence to the manifold $\mathcal M$ and splitting into two bubbles. Together with the surprising findings from Theorem \ref{theorem Bianchi Egnell new}, this testifies to the remarkably rich structure of stability inequalities in the reverse Sobolev setting.

\subsection{Notation}

We denote $(\int_{\mathbb S^n} u^q)^\frac{1}{q} =: \|u\|_q$ for every $q \in \R \setminus \{0\}$. 

We abbreviate $\frac{2n}{n-2s} =: p$. Note that $p < 0$ if $s > \frac{n}{2}$, which is the case we consider here. 

For $k \in \N_0$, we let $E_k$ (resp. $E_{\geq k}$) be the subspace of $L^2(\mathbb S^n)$ spanned by spherical harmonics of degree equal to (resp. greater or equal to) $k$. 

The measure $d \omega$ is the (non-normalized) measure on $\Sph^n$ induced by the standard scalar product of $\R^{n+1}$. We sometimes omit $d \omega$ for brevity when writing integrals over $\mathbb S^n$ (or $dx$ when writing integrals over $\R^n$).

The notation $(\cdot, \cdot)$ denotes the scalar product of $L^2(\mathbb S^n)$.

\section{The distance $\mathsf d(u)$}

The key property needed to make the definition of $\mathsf d(u)$ in \eqref{d(u) definition} a valid one is that the infimum in \eqref{d(u) definition} is not taken over the empty set. This is ensured by the following proposition, which is the main result we prove in this section. 

\begin{proposition}
\label{proposition balance condition}
Suppose that $s - \frac{n}{2} \in (0,1) \cup (1,2)$. Let $u \in {H}^s(\mathbb S^n)$ with $u > 0$. Then there exists $h \in \mathcal M$ such that $u = h + \rho$ with $\rho \in (T_h \mathcal M)^\perp$. Moreover, $a_{2s}$ is positive definite on $(T_h \mathcal M)^\perp$. 
\end{proposition}

\subsection{Preliminaries}

We collect here some important definitions and properties, which will be used freely in the following. 

\textit{Stereographic projection.  } We recall that the inverse stereographic projection $\mathcal S: \R^n \to \mathbb S^n$ with center $S = - e_{n+1}$ is given by (see, e.g., \cite[Section 4.4]{Lieb2001})
\[ \mathcal S_i(x) = \frac{2x_i}{1 + |x|^2}, \qquad \mathcal S_{n +1}(x) = \frac{1 - |x|^2}{1 + |x|^2}, \qquad J_{\mathcal S}(x) = \left(\frac{2}{1 + |x|^2} \right)^n , \]
where for a differentiable map $f$ we abbreviate $J_{f}(x) := |\det D f(x)|$. 
Its inverse $\mathcal S^{-1}: \mathbb S^n \setminus \{S\} \to \R^n$ is given by 
\[ (\mathcal S^{-1}(\omega))_i = \frac{\omega_i}{1 + \omega_{n+1}}, \qquad J_{\mathcal S^{-1}}(\omega)  = (1 + \omega_{n+1})^{-n}. \]

For many purposes below it is convenient to introduce the following normalized representatives of the family $\mathcal M$ of reverse Sobolev optimizers. For every $\zeta \in \R^{n+1}$ with $|\zeta| < 1$, we let
\[ v_\zeta(\omega) := (1 - |\zeta|^2)^{-\frac{2s-n}{4}} (1 - \zeta \cdot \omega)^{-\frac{2s-n}{2}} \]
Hence $v_\zeta \in \mathcal M$, and the normalization factor $(1 - |\zeta|^2)^{-\frac{2s-n}{4}}$ is chosen such that 
\begin{equation}
    \label{v zeta normalization}
    A_{2s} v_\zeta = \alpha_{2s}(0) v_\zeta^{p-1}, \qquad  a_{2s}[v_\zeta] = a_{2s}[1] = \alpha_{2s}(0) |\mathbb S^n| \quad \text{ and } \quad \int_{\mathbb S^n} v_\zeta^p  = \int_{\mathbb S^n} 1^p = |\mathbb S^n|. 
\end{equation}

\textit{Conformal invariance.  } 
Given $u \in H^s(\mathbb S^n)$, and a conformal map $\Phi$ from $\mathbb S^n$ to itself, we define 
\begin{equation}
\label{uT definition}
u_\Phi(\omega) := u(\Phi(\omega)) J_\Phi(\omega)^\frac{n-2s}{2n}. 
\end{equation} 
It follows directly from the definition that $\|u_\Phi\|_{\frac{2n}{n-2s}} = \|u\|_\frac{2n}{n-2s}$ for every $0 < u \in H^s(\mathbb S^n)$. Moreover, by the conformal invariance of $a_{2s}$ \cite[Lemma 3]{FrKoTa2022}, we also have
\begin{equation}
\label{a2s conf inv}
a_{2s}[u] = a_{2s}[u_\Phi] \quad \text{ for every } \, u \in H^s(\mathbb S^n). 
\end{equation}

\textit{Reverse Hölder inequality.  } The \emph{reverse Hölder inequality} states that for functions $0 < f, g \in C(\mathbb S^n)$  and every $ q \in (1, \infty)$ one has
\begin{equation}
    \label{reverse hölder}
    \int_{\mathbb S^n} |fg| \geq \|f\|_{1/q} \|g\|_{- \frac{1}{q-1}} \qquad \text{ with equality if and only if } \quad  |f| = c |g|^\frac{-q}{q-1}. 
\end{equation} 
Inequality \eqref{reverse hölder} follows from the usual Hölder inequality by writing
\[ \|f\|_\frac{1}{q}^\frac{1}{q} = \|f^\frac{1}{q}\|_1 = \|(fg)^\frac{1}{q} g^{-\frac{1}{q}} \|_1 \leq \|(fg)^\frac{1}{q}\|_q \|g^{-\frac{1}{q}}\|_\frac{q}{q-1} = \|fg\|_1 \|g\|_{-\frac{1}{q-1}}^{-\frac{1}{q}}. \]
Raising to the $q$-th power and rearranging terms, we obtain the inequality in \eqref{reverse hölder}. Moreover, the equality condition for the standard Hölder inequality translates to that stated in \eqref{reverse hölder}.

\subsection{The proof of Proposition \ref{proposition balance condition}}

For the proof of Proposition \ref{proposition balance condition} we take inspiration from the argument in \cite[Appendix B]{Frank2012} (which is fundamentally the same as that sketched in \cite[Section 4]{Ha}). Even if no further difficulties arise, we choose to give full details here, because of the central importance to our construction and because none of the mentioned references treats the reverse case explicitly.

Let $\xi \in \mathbb S^n$ and fix an orthogonal matrix $O_\xi \in O(n+1)$ such that $O_\xi \xi = N = e_{n+1}$. For $\delta > 0$ we then consider the conformal map from $\mathbb S^n$ to itself given by 
\[ \gamma_{\delta, \xi} = O_\xi^T \circ \mathcal S \circ D_\delta \circ  \mathcal S^{-1} \circ O_\xi, \]
where $D_\delta(x) = \delta x$ denotes dilation on $\R^n$.

Here is the crucial property satisfied by the family of maps $\gamma_{\delta, \xi}$. 

\begin{lemma}
\label{lemma balance condition}
Suppose that $s - \frac{n}{2} \in (0,1) \cup (1,2)$. Let $u \in C(\mathbb S^n)$ with $u > 0$ on $\mathbb S^n$. Then there exists $\delta \in [1, \infty)$ and $\xi \in \mathbb S^n$ such that 
\begin{equation}
    \label{balance condition lemma}
    \int_{\mathbb S^n} \gamma_{\delta, \xi}(\omega) J_{\gamma_{\delta, \xi}}(\omega)^\frac{n+2s}{2n} u(\omega) \, d \omega = 0. 
\end{equation} 
\end{lemma}

From this Proposition \ref{proposition balance condition} follows immediately. 

\begin{proof}
[Proof of Proposition \ref{proposition balance condition}]
Changing variables $\omega \mapsto \gamma_{\delta, \xi}^{-1}(\omega)$ in \eqref{balance condition lemma} we obtain 
\[ 0 = \int_{\mathbb S^n} \omega J_{\gamma_{\delta, \xi}^{-1}}(\omega)^\frac{n-2s}{2n} u(\gamma_{\delta, \xi}^{-1}(\omega)) \, d \omega. \]
Recalling \eqref{uT definition} and defining $\Phi = \gamma_{\delta, \xi}$ this says precisely that there is $c \in \R$ and $\rho \in E_{\geq 2} = (T_c \mathcal M)^\perp$ such that 
\[ u_{\Phi^{-1}} = c + \rho. \]
In other words, $u_{\Phi^{-1}}$ satisfies the center of mass condition $\int_{\mathbb S^n} \omega u_{\Phi^{-1}} = 0$ from \cite{Ha, Gong2025}. 
Applying $\Phi$, we obtain $u = c_\Phi + \rho_\Phi$. Since $\rho_\Phi \in (T_{c_\Phi} \mathcal M)^\perp$ by Lemma \ref{lemma tangent space transformation}, this is the desired decomposition.

Moreover, we have $a_{2s}[\rho] \geq 0$ with equality only if $\rho \equiv 0$, because $a_{2s}$ is positive definite on $(T_c \mathcal M)^\perp = E_{\geq 2}$. By Lemma \ref{lemma tangent space transformation} and the conformal invariance \eqref{a2s conf inv}, $a_{2s}$ is positive definite on $(T_h \mathcal M)^\perp$ for every $h \in \mathcal M$. 
\end{proof}

It remains to prove Lemma \ref{lemma balance condition}. 

\begin{proof}
[Proof of Lemma \ref{lemma balance condition}]
Let us denote by $F(\delta, \xi) := \int_{\mathbb S^n} \gamma_{\delta, \xi}(\omega) J_{\gamma_{\delta, \xi}}(\omega)^\frac{n+2s}{2n} u(\omega) \, d \omega$ the quantity on the left side of \eqref{balance condition lemma}. To evaluate this integral, it is convenient to pass to $\R^n$. Changing variables $\omega = O_\xi^T (\mathcal S(x))$, and defining 
\begin{equation}
    \label{v xi definition}
    v_\xi(x) := J_{O_\xi^T \circ \mathcal S}(x)^{1/p} u(O_\xi^T \mathcal S(x)) = \left( \frac{1+|x|^2}{2} \right)^\frac{2s-n}{2}  u(O_\xi^T \mathcal S(x)), 
\end{equation} 
we obtain
\begin{align*}
    F(\delta, \xi) &= \int_{\R^n} O_\xi^T \mathcal S(\delta x) J_{O_\xi^T \circ \mathcal S \circ D_\delta}(x)^\frac{n+2s}{2n} v_\xi(x) \, dx \\
    &= O_\xi^T \left( \int_{\R^n} \mathcal S(\delta x) \left(  \frac{2 \delta}{1  + \delta^2 |x|^2}\right)^\frac{n+2s}{2}  v_\xi(x) \, dx \right). 
\end{align*}
Let us evaluate the components of this vector-valued integral. For $i = 1,...,n$ we have 
\begin{align*}
    & \qquad \int_{\R^n} \mathcal S(\delta x)_i \left(  \frac{2 \delta}{1  + \delta^2 |x|^2}\right)^\frac{n+2s}{2}  v_\xi(x) \, dx \\
    &= \int_{\R^n} \frac{\delta x_i}{1 + \delta^2 |x|^2} \left(  \frac{2 \delta}{1  + \delta^2 |x|^2}\right)^\frac{n+2s}{2}  v_\xi(x) \, dx \\
    &= \delta^{- \frac{n-2s}{2}} \int_{\R^n} \frac{x_i}{1 + |x|^2} \left(  \frac{2}{1  +  |x|^2}\right)^\frac{n+2s}{2}  v_\xi(\frac{x}{\delta}) \, dx \\
    &= \delta^{- \frac{n-2s}{2}} \left( v_\xi(0) \int_{\R^n} \frac{x_i}{1 + |x|^2} \left(  \frac{2}{1  +  |x|^2}\right)^\frac{n+2s}{2}  \, dx + o(1) \right)  = o (\delta^{-\frac{n-2s}{2}})
 \end{align*}
 as $\delta \to \infty$, by dominated convergence (using \eqref{v xi definition} and boundedness of $u$) and antisymmetry of the integrand. For $i = n+1$ we obtain similarly 
\begin{align*}
    & \qquad \int_{\R^n} \mathcal S(\delta x)_{n+1} \left(  \frac{2 \delta}{1  + \delta^2 |x|^2}\right)^\frac{n+2s}{2}  v_\xi(x) \, dx \\
    &= \int_{\R^n} \frac{1 - \delta^2 |x|^2}{1 + \delta^2 |x|^2} \left(  \frac{2 \delta}{1  + \delta^2 |x|^2}\right)^\frac{n+2s}{2}  v_\xi(x) \, dx \\
    &= \delta^{- \frac{n-2s}{2}} \int_{\R^n} \frac{1 - |x|^2}{1 + |x|^2} \left(  \frac{2}{1  +  |x|^2}\right)^\frac{n+2s}{2}  v_\xi(\frac{x}{\delta}) \, dx = \delta^{- \frac{n-2s}{2}} \left( c_0 v_\xi(0)  + o(1) \right),
 \end{align*}
 as $\delta \to \infty$, where 
 \[ c_0 := \int_{\R^n} \frac{1 - |x|^2}{1 + |x|^2} \left(  \frac{2}{1  +  |x|^2}\right)^\frac{n+2s}{2}  \, dx = 2^{\frac{n-2s-2}{2}} |\mathbb S^{n-1}| \frac{\Gamma(n/2) \Gamma(s)}{\Gamma(\frac{n+2s+2}{2})} \left(s - \frac{n}{2}\right) > 0. \]
 Since $O_\xi^T e_{n+1} = \xi$ and $v_\xi(0) = 2^\frac{n-2s}{2} u(O_\xi^T e_{n+1}) = 2^\frac{n-2s}{2} u(\xi) > 0$, we have shown that 
 \[ G(\delta, \xi) :=  \delta^\frac{n-2s}{2} \frac{1}{c_0 u(\xi)} F(\delta, \xi) \to \xi \]
 as $\delta \to \infty$, for every $\xi \in \mathbb S^n$. This means that the function $H: B^{n+1} \to \R^{n+1}$, defined on the open unit ball $B^{n+1} \subset \R^{n+1}$ by
 \[ H(r\xi) = G\left(\frac{1}{1-r}, \xi\right), \qquad r \in [0,1), \, \xi \in \mathbb S^n, \]
 extends continuously to the closed unit ball $\overline{B^{n+1}}$ and is the identity on the boundary. As a well-known consequence of Brouwer's fixed point theorem, $H$ must have a zero $r \xi$ in $B^{n+1}$. Thus $F(\frac{1}{1-r}, \xi) = 0$, and \eqref{balance condition lemma} follows with $\delta = \frac{1}{1-r}$.  
\end{proof}

\subsection{Some more properties of $\mathsf d(u)$}

Here are some additional properties which show that $\mathsf d(u)$ is well-behaved.

We first show that the tangent spaces transform well under conformal mappings.

\begin{lemma}
\label{lemma tangent space transformation}
Let $s - \frac{n}{2} \in (0,1) \cup (1,2)$. Let $h \in \mathcal M$, and let $\Phi: \mathbb S^n \to \mathbb S^n$ be conformal. Then $\rho \in (T_h \mathcal M)^\perp$ if and only if $\rho_\Phi \in (T_{h_\Phi} \mathcal M)^\perp$. 
\end{lemma}

\begin{proof}
We show equivalently that $\eta \in T_h \mathcal M$ if and only if $\eta_\Phi \in T_{h_\Phi} \mathcal M$. From this the assertion follows easily using conformal invariance of $a_{2s}$. 

If $\eta \in T_h \mathcal M$, then there is a path $(- \eps, \eps) \ni t \mapsto \gamma_t \in \mathcal M$ such that $\gamma_0 = h$ and 
\[ \eta(\omega) = \lim_{t \to 0} \frac{\gamma_t(\omega) - h(\omega)}{t}. \]
But then the path defined by $\tilde{\gamma}_t(\omega) := J_\Phi(\omega)^{1/p} \gamma_t(\Phi\omega)$ is again in $\mathcal M$, with $\tilde{\gamma}_0 = h_\Phi$. Thus 
\[ \eta_\Phi(\omega) =  J_\Phi(\omega)^\frac{1}{p}  \lim_{t \to 0} \frac{\gamma_t(\Phi\omega) - h(\Phi\omega)}{t} = \lim_{t \to 0} \frac{\tilde{\gamma}_t(\omega) - h_\Phi(\omega)}{t} \]
lies in $T_{h_\Phi} \mathcal M$. 

The converse implication is obtained through the same argument, by replacing $\Phi$ by $\Phi^{-1}$. 
\end{proof}

\begin{lemma}
\label{lemma d(u)}
Let $s - \frac{n}{2} \in (0,1) \cup (1,2)$. Let $u \in H^s(\mathbb S^n) \setminus \mathcal M$ with $u > 0$ and let $\mathsf d(u)$ be defined by \eqref{d(u) definition}. Then the infimum defining $\mathsf d(u)$ is achieved. 
\end{lemma}

\begin{corollary}
Let $s - \frac{n}{2} \in (0,1) \cup (1,2)$.  For every $0 < u \in H^s(\mathbb S^n) \setminus \mathcal M$, we have $\mathsf d(u) > 0$. 
\end{corollary}

\begin{proof}
Since $\mathsf d(u)$ is achieved by Lemma \ref{lemma d(u)}, we have $\mathsf d(u) = a_{2s}[u-h]$ for some $h \in \mathcal M$. Since $0 \neq u- h \in (T_h \mathcal M)^\perp$ and since $a_{2s}$ is positive definite on $(T_h \mathcal M)^\perp$ by Proposition \ref{proposition balance condition}, we obtain $\mathsf d(u) = a_{2s}[u-h]  >0$.
\end{proof}

\begin{proof}[Proof of Lemma \ref{lemma d(u)}]

Let $h_k = c_k v_{\zeta_k}$ be a minimizing sequence for $\mathsf d(u)$. 

For any $k$ the orthogonality condition $u - h_k \in (T_{h_k} \mathcal M)^\perp$ implies 
\begin{equation}
    \label{orth proof}
    0 = (u - h_k, A_{2s} h_k) = c_k \alpha_{2s}(0) \int_{\mathbb S^n} u v_{\zeta_k}^{p-1} - c_k^2 \alpha_{2s}(0) |\mathbb S^n|, 
\end{equation} 
that is, 
\begin{equation}
    \label{c proof}
    c_k = \frac{1}{|\mathbb S^n|} \int_{\mathbb S^n} u v_{\zeta_k}^{p-1}. 
\end{equation} 

 Applying the reverse Hölder inequality \eqref{reverse hölder} with $p = -\frac{1}{q-1} \Leftrightarrow q = \frac{p-1}{p}$, we get 
\begin{equation}
    \label{reverse hölder bound}
    \int_{\mathbb S^n} u v_{\zeta_k}^{p-1} \geq \|u\|_p \|v_{\zeta_k}\|_p^{p-1} = \|u\|_p |\mathbb S^n|^{\frac{p-1}{p}}. 
\end{equation}  
Hence $c_k \geq \|u\|_p |\mathbb S^n|^{-\frac{1}{p}}$ independently of $k$. 

On the other hand, we can write, using \eqref{orth proof}, 
\[ \mathsf d(u) + o(1) = a_{2s}[u - h_k] = a_{2s}[u] - a_{2s}[h_k] = a_{2s}[u] - \alpha_{2s}(0) c_k^2 |\mathbb S^n|, \]
which implies $c_k^2 = \frac{a_{2s}[u] - \mathsf d(u) + o(1)}{\alpha_{2s}(0) |\mathbb S^n|}$, and hence $c_k$ is also bounded from above. From \eqref{c proof} also $\int_{\mathbb S^n} u v_{\zeta_k}^{p-1}$ is bounded, and hence $|\zeta_k|$ remains bounded away from 1, by Lemma \ref{lemma int u v^{p-1}} below. 

We have thus shown that, up to taking a subsequence, there are $c > 0$ and $\zeta \in \R^{n+1}$ with $|\zeta| < 1$ such that $c_k \to c$ and $\zeta_k \to \zeta$. Then clearly $h_k \to h := c v_\zeta$ in $H^s$, and $h$ is the desired minimizer of $\mathsf d(u)$. 
\end{proof}

\section{Stability of the reverse Sobolev inequality}
\label{section stability}

In this section we give the proof of Theorem \ref{theorem Bianchi Egnell new}. In an adaptation of the classical strategy of Bianchi and Egnell \cite{BiEn}, the proof rests on the following two main facts. 

\begin{proposition}
\label{proposition local 2}
Suppose that $s - \frac{n}{2} \in (0,1) \cup (1,2)$. 
Let $(\rho_k)_{k \in \N} \subset E_{\geq 2}$ be such that $\|\rho_k\|_{H^s} \to 0$ as $k \to \infty$, and let $u_k  = 1 + \rho_k$. 
\begin{enumerate}[(i)]
    \item If $s - \frac{n}{2} \in (0,1)$, then, as $k \to \infty$,
\[ \mathcal E(u_k) \geq \frac{4s }{n+2s+2} + o(1), \]
with asymptotic equality if and only if there is $ \rho \in E_2$ such that $\frac{\rho_k}{\|\rho_k\|_{H^s}} \to \rho$ in $H^s$, as $k \to \infty$. 
\item If $s - \frac{n}{2} \in (1,2)$, then, as $k \to \infty$, 
\[  \mathcal E(u_k) \geq 1 + o(1), \]
with asymptotic equality if and only if 
\[ \int_{\mathbb S^n} \rho_k^2 = o(a_{2s}[\rho_k]) \qquad \text{ as } k \to \infty. \]
\end{enumerate}
\end{proposition}

\begin{proposition}
\label{proposition compactness 2}
Let $s - \frac{n}{2} \in (0,1) \cup (1,2)$. 
Suppose that $(u_k) \subset H^s(\mathbb S^n) \setminus \mathcal M$ is a sequence such that $u_k >0$ and $\mathcal E(u_k) \to 0$ as $n \to \infty$. Then there are constants $c_k > 0$ and conformal transformations $\Phi_k$ such that $c_k (u_k)_{\Phi_k} \to 1$ strongly in $H^s(\mathbb S^n)$. 
\end{proposition}

We also need the analogue of \cite[Proposition 7]{Bahri1988}. 

\begin{proposition}
\label{proposition bahri-coron}
Let $s - \frac{n}{2} \in (0,1) \cup (1,2)$. There are neighborhoods $U, V$ of $1$ in $H^s(\mathbb S^n)$ 
such that the following holds: For every $u \in U$ there is a unique $h \in V \cap \mathcal M$ such that $\rho := u - h \in (T_h \mathcal M)^\perp$. Moreover, $\mathsf d(u) = a_{2s}[\rho]$. 
\end{proposition}

From these three facts, Theorem \ref{theorem Bianchi Egnell new} follows easily. 
\begin{proof}[Proof of Theorem \ref{theorem Bianchi Egnell new}]
By contradiction, if $c_{BE}(s) = 0$, then there exists $(u_k)$ in $H^s(\mathbb S^n)$ with $u_k> 0$ such that $\mathcal E(u_k) \to 0$. By Proposition \ref{proposition compactness 2}, we have that $v_k = c_k (u_k)_{\Phi_k} \to 1$ in $H^s(\mathbb S^n)$, for certain $(c_k) \subset \R_+$ and conformal maps $\Phi_k$, and still $\mathcal E(v_k) \to 0$ by conformal invariance. By Lemma \ref{lemma d(u)} the infimum defining $\mathsf d(v_k)$ is achieved. Hence there are conformal maps $S_k$ such that  
\[ (v_k)_{S_k} = d_k + \eta_k, \]
with $d_k > 0$,  $\eta_k \in E_{\geq 2}$ and $\mathsf d(v_k) = a_{2s}[\eta_k]$. By Proposition \ref{proposition bahri-coron}, $(v_k)_{S_k}$ lies in a neighborhood of $1$ for every $k$. Since $v_k \to 1$, we must have $d_k \to 1$ and $\|\eta_k\|_{H^s(\mathbb S^n)} \to 0$. Now setting $\rho_k := d_k^{-1} \eta_k$, the function 
\[ w_k := d_k^{-1} (v_k)_{S_k} = 1 + \rho_k \]
satisfies the assumption of Proposition \ref{proposition local 2} and we conclude 
\[
\liminf \mathcal E(u_k) = \liminf \mathcal E(w_k) \geq \min \left\{ \frac{4s}{n + 2s + 2}, 1 \right\} > 0.
\] 
This is the desired contradiction to $\mathcal E(u_k) = o(1)$, so we have proved $c_{BE}(s) > 0$.  

The inequalities $c_{BE}(s) \leq \frac{4s}{n+2s+2}$ for $s - \frac{n}{2} \in (0,1)$ and  $c_{BE}(s) \leq 1$ claimed in (a) and (b) are direct consequences of Proposition \ref{proposition local 2}. To complete the proof of (b), it suffices to use the stability result from \cite{Gong2025}. Indeed, in the case $s - \frac{n}{2} \in (1,2)$, \cite[Theorem 1.2]{Gong2025} can be rephrased by saying that for every $0 < u \in H^s(\mathbb S^n)$, there exists $\Phi$ conformal such that $u_\Phi = c + \rho$ with $\rho \in E_{\geq 2}$ and 
\begin{equation}
\label{gong stab}
a_{2s}[u] - \mathcal S_s \|u\|_p^2 \geq a_{2s}[\rho]. 
\end{equation} 
This means that we can write $u = c_{\Phi^{-1}} + \rho_{\Phi^{-1}}$ with $\rho_{\Phi^{-1}} \in (T_{c_{\Phi^{-1}}} \mathcal M)^\perp$ by Lemma \ref{lemma tangent space transformation}. Recalling the definition of \eqref{d(u) definition} and \eqref{a2s conf inv}, we see that $a_{2s}[\rho] = a_{2s}[\rho_\phi] \geq \mathsf d(u)$. Together with \eqref{gong stab}, we find $c_{BE}(s) \geq 1$, and hence $c_{BE}(s) = 1$. 

Finally, an inspection of the proof of \cite[Theorem 1.2]{Gong2025} shows that for equality to hold in \eqref{gong stab}, one must have equality in the reverse Hölder inequality \cite[eq. (3.12)]{Gong2025}
\[ \int_{\mathbb S^n} u_\Phi \cdot 1 \geq \|u_\Phi\|_\frac{2n}{n-2s} \left( \int_{\mathbb S^n} 1 \right)^\frac{n+2s}{2n}. \]
By the equality conditions in the reverse Hölder inequality (see \eqref{reverse hölder}), this forces $u_\Phi = c$ for some $c > 0$, and hence $u \in \mathcal M$. Hence inequality \eqref{gong stab} is indeed sharp for every $0<u \in H^s(\mathbb S^n) \setminus \mathcal M$. It follows that $c_{BE}(s) = 1$ does not admit a minimizer when $s - \frac{n}{2} \in (1,2)$. 
\end{proof}

\subsection{Proof of the auxiliary statements}

It remains to prove the assertions from Propositions \ref{proposition local 2},  \ref{proposition compactness 2} and  \ref{proposition bahri-coron}. 

We give the proof of Proposition \ref{proposition bahri-coron} in Section \ref{subsection bahri-coron} below. In turn, this proposition will be used in the proof of Proposition \ref{proposition local 2}, which we give now.

\begin{proof}
[Proof of Proposition \ref{proposition local 2}]
Since $\|\rho_k\|_{H^s} \to 0$, by Proposition \ref{proposition bahri-coron} we have $\mathsf d(u_k) = a_{2s}[\rho_k]$. Therefore we can write 
\[ \mathcal E(u_k) = \frac{a_{2s}[1] + a_{2s}[\rho_k] - \mathcal S_s \left(\int_{\mathbb S^n} (1 + \rho_k)^p \right)^{2/p}}{a_{2s}[\rho_k]}. \]
Using that $\|\rho_k\|^2_\infty \lesssim \|\rho_k\|^2_{H^s} \sim a_{2s}[\rho_k] \to 0$ by the embedding $H^s(\mathbb S^n) \to C(\mathbb S^n)$, a Taylor expansion yields 
\[ (1 + \rho_k)^p = 1 + p \rho_k + \frac{p(p-1)}{2} \rho_k^2 + \mathcal O(|\rho_k|^3) \]
pointwise, and thus, using $\int_{\mathbb S^n} \rho_k = 0$,
\[ \left( \int_{\mathbb S^n} (1 + \rho_k)^p \right)^{2/p}  = \int_{\mathbb S^n} 1 + (p-1) \left( \int_{\mathbb S^n} 1 \right)^{\frac{2}{p} - 1} \int_{\mathbb S^n} \rho_k^2 + o(a_{2s}[\rho_k]). \]
In total, using $a_{2s}[1] = \mathcal S_s (\int_{\mathbb S^n} 1)^{\frac{2}{p}}$ and $(p-1) \mathcal S_s \left( \int_{\mathbb S^n} 1 \right)^{\frac{2}{p} - 1} = (p-1) \alpha_{2s}(0) = \alpha_{2s}(1)$, we obtain  
\[ \mathcal E(u_k)  = \frac{a_{2s}[\rho_k](1 + o(1)) - \alpha_{2s}(1) \int_{\mathbb S^n} \rho_k^2 }{a_{2s}[\rho_k]} = 1 - \alpha_{2s}(1) \frac{\int_{\mathbb S^n}\rho_k^2 }{a_{2s}[\rho_k]} + o(1). \]
Now we need to distinguish two cases. If $s - \frac{n}{2} \in (0,1)$, then $\alpha_{2s}(1) > 0 $. Together with $a_{2s}[\rho_k] \geq \alpha_{2s}(2)\int_{\mathbb S^n} \rho_k^2$ (because $\rho_k \in E_{\geq 2}$) we get
\[ \mathcal E(u_k) \geq 1 - \frac{\alpha_{2s}(1)}{\alpha_{2s}(2)} + o(1) = \frac{4s}{n + 2s +2} + o(1). \]
Equality holds (asymptotically) if and only if $a_{2s}[\rho_k] = \alpha_{2s}(2)\int_{\mathbb S^n} \rho_k^2 + o(a_{2s}[\rho_k])$, which translates to the condition in the statement. 

If however $s - \frac{n}{2} \in (1,2)$, then $\alpha_{2s}(1) < 0$ and we obtain the estimate 
\[ \mathcal E(u_k) \geq 1 + o(1).  \]
Now equality holds (asymptotically) if and only if $\int_{\mathbb S^n} \rho_k^2 = o(a_{2s}[\rho_k])$. 
\end{proof}

\begin{proof}
[Proof of Proposition \ref{proposition compactness 2}]
Let $(u_k)_{k \in \N}$ be a sequence such that $\mathcal E(u_k) \to 0$ as $k \to \infty$. Up to multiplication by a constant and a conformal transformation, we may assume that 
\[ u_k = 1 + \rho_k, \qquad \text{ with } \mathsf d(u_k) = a_{2s}[\rho_k]. \]
Dividing the numerator and denominator of the quotient by $\|u_k\|_p^2$, we get
\[ o(1) = \mathcal E(u_k) = \frac{\frac{a_{2s}[u_k]}{\|u_k\|_p^2} - \mathcal S_s}{ \frac{\mathsf d(u_k)}{\|u_k\|_p^2}}. \]
We claim that there is $C > 0$ such that 
\begin{equation}
    \label{claim bound Q/p}
     \frac{\mathsf d(u_k)}{\|u_k\|_p^2} \leq C  
\end{equation}
uniformly in $k$. If this is true, it follows that $(u_k)$ is a minimizing sequence for the reverse Sobolev quotient $\frac{a_{2s}[u_k]}{\|u_k\|_p^2}$. Then \cite[Proposition 6]{FrKoTa2022} asserts that there are $c_k > 0$ and conformal transformations $\Phi_k$ such that $c_k (u_k)_{\Phi_k} \to h$ strongly in $H^s(\mathbb S^n)$, for some $h \in \mathcal M$. Since $h$ is itself of the form $h = c 1_\Phi$ for some $c > 0$ and some conformal map $\Phi$, we get $c^{-1} c_k (u_k)_{\Phi^{-1} \circ \Phi_k} \to 1$ in $H^s(\mathbb S^n)$, as we desired to show. 

So it remains to show \eqref{claim bound Q/p}. We start by noting that because of $a_{2s}[\rho_k] = \mathsf d(u_k)$ the assumption $\mathcal E(u_k) \to 0$ can be written as 
\begin{equation}
    \label{claim Q/p proof identity}
    a_{2s}[1] + (1 + o(1)) \mathsf d(u_k) - \mathcal S_s \|u_k\|_p^2 = 0, 
\end{equation} 
that is, 
\[ (1 + o(1)) \frac{\mathsf d(u_k)}{\|u_k\|_p^2} = \mathcal S_s - \frac{a_{2s}[1]}{\|u_k\|_p^2}.  \]
If $s - \frac{n}{2} \in (1,2)$, then $a_{2s}[1] \geq 0$, from which \eqref{claim bound Q/p} follows immediately. 

If $s - \frac{n}{2} \in (0,1)$, we need to argue slightly differently. We return to \eqref{claim Q/p proof identity} and divide by $\mathsf d(u_k)$ to get 
\begin{equation}
    \label{claim Q/p proof identity 2}
0 = \frac{a_{2s}[1]}{\mathsf d(u_k)} + (1 + o(1)) - \mathcal S_s \frac{\|u_k\|_p^2}{\mathsf d(u_k)}.
\end{equation} 
Since $\mathcal S_s < 0$ in this case, we find that 
\begin{equation}
\label{a2s bound proof}
 \frac{a_{2s}[1]}{\mathsf d(u_k)} \leq -1+o(1),
\end{equation}
i.e., $\mathsf d(u_k) \leq C$. 
We can now argue by contradiction: If \eqref{claim bound Q/p} is not true, then as $k \to \infty$ the last term in \eqref{claim Q/p proof identity 2} goes to zero, up to extracting a subsequence (we do not mention further subsequences in the following argument). Since $\mathsf d(u_k) = a_{2s}[\rho_k]$ is bounded, $\rho_k$ is bounded in $H^s(\mathbb S^n)$. Thus there is $\rho \in H^s(\mathbb S^n)$ such that $\rho_k \rightharpoonup \rho$ weakly in $H^s(\mathbb S^n)$. Passing to the limit in \eqref{a2s bound proof}, we obtain 
\[ a_{2s}[1] = - \lim_{k \to \infty} a_{2s}[\rho_k] + o(1) \leq - a_{2s}[\rho] \]
by lower semicontinuity of $a_{2s}$ (see \cite[proof of Proposition 6]{FrKoTa2022}). Since $a_{2s}[\rho, 1] = 0$, we can rewrite this as 
\begin{equation}
    \label{1 + rho leq 0}
    a_{2s}[1+\rho] \leq 0. 
\end{equation} 

On the other hand, since $\mathsf d(u_k)$ is bounded, for $\frac{\mathsf d(u_k)}{\|u_k\|_p^2} \to \infty$ to be true, we need $\|u_k\|_p^2 \to 0$, which forces that $\inf_{\mathbb S^n} (1  +\rho_k) \to 0$ as $k \to \infty$. By the compact embedding of $H^s(\mathbb S^n)$ into $C(\mathbb S^n)$, we have $\rho_k \to \rho$ uniformly on $\mathbb S^n$, so up to a rotation we may assume that $(1 + \rho)(S) =0$. Under this hypothesis \cite[Proposition 5]{FrKoTa2022} asserts that 
\[ a_{2s}[1+\rho] \geq 0,\]
with equality if and only if 
\begin{equation}
    \label{1 + rho optimizer}
    1 + \rho(\omega) = c ( 1+ \omega_{n+1})^\frac{2s-n}{2}. 
\end{equation}
for some $c \in \R$. 
In conclusion, together with \eqref{1 + rho leq 0} we obtain that $1 + \rho$ must be of the form \eqref{1 + rho optimizer}. By integrating \eqref{1 + rho optimizer} (against $1$), we find that $c > 0$. 

Now we can obtain the desired contradiction. Namely, integrating \eqref{1 + rho optimizer} against $\omega_{n+1}$, recalling that $\rho \in E_{\geq 2}$, we get 
\[ 0 = \int_{\mathbb S^n} (1 + \rho) \omega_{n+1} = c \int_{\mathbb S^n} ( 1+ \omega_{n+1})^\frac{2s-n}{2} \omega_{n + 1}. \]
But clearly, the integral on the right side is strictly positive, which gives a contradiction and finishes the proof.  
\end{proof}

 \subsection{The proof of Proposition \ref{proposition bahri-coron}}
 \label{subsection bahri-coron}

The proof of Proposition \ref{proposition bahri-coron} consists in applying the implicit function theorem to the function 
\begin{align*}
    F = (F_0, ..., F_{n+1}): \R \times B_{\R^{n+1}}(0,1) \times H^s(\mathbb S^n) & \to \R^{n+2}
\end{align*}
given by 
\[ F_0(c, \zeta, u)= (u - c v_\zeta, A_{2s} v_\zeta) \]
and 
\[ F_i(c, \zeta ,u) = (u - c v_\zeta, A_{2s} \partial_{\zeta_i} v_\zeta), \qquad i = 1,...,N+1. \]
Namely, we have $F(1, 0, 1) = 0$ and moreover simple computations show 
\[ \partial_c F_0(1,0,1) = - a_{2s}[1] \neq 0, \qquad \partial_{\zeta_i}F_0(1,0,1) = - \alpha_{2s}(0) \int_{\mathbb S^n} \omega_i  = 0 \]
and 
\[ \partial_{\zeta_i} F_j(1,0,1) = \alpha_{2s}(1) (1 - \omega_i, \omega_j) = -\frac{\alpha_{2s}(1) }{n+1} |\mathbb S^n| \delta_{ij}. \]
Hence $D_{(c,\zeta)} F(1,0,1)$ is invertible. Thus the implicit function theorem yields the existence of neighborhoods $U$ of $1$ in $H^s(\mathbb S^n)$ and $V$ of $(1,0)$ in $\R \times B(0,1)$ and a function $h: U \to V$ such that for $u \in U$ and $(c, \zeta) \in V$ one has $F(c, \zeta, u) = 0$ if and only if $(c, \zeta) = h(u) =: (c(u), \zeta(u))$. Since  the $(c, \zeta) \in V$ parametrize precisely a neighborhood of $1$ in $\mathcal M$ (via $(c, \zeta) = c v_\zeta$), and since by definition $F(c, \zeta, u) = 0$ if and only if $u - c v_\zeta \in (T_{c v_\zeta} \mathcal M)^\perp$, we obtain the statement of Proposition \ref{proposition bahri-coron}. 

It remains to prove that, up to making $U$ and $V$ smaller, we have $\mathsf d(u) = a_{2s}[\rho]$ for every $u \in U$, where $\rho = u - c v_{\zeta}$ and we write $(c, \zeta) = h(u)$.  By contradiction, if this is not true, then there are sequences $(u_k) \subset H^s(\mathbb S^n)$ and $(h_k) = c_k v_{\zeta_k} \subset \mathcal M$ such that $u_k - h_k \in (T_{h_k} \mathcal M)^\perp$, 
\[ u_k \to 1 \, \text{ in } H^s, \qquad |c_k - 1| + |\zeta_k| > \delta_0 \, \text{ for some } \delta_0 > 0 \]
and such that for all $k \in \N$
\[ a_{2s}[u_k-h_k]  < a_{2s}[u_k-h(u_k)], \]
where $h(u_k)$ is the function associated to $u_k$ through the implicit function theorem as described above. Since $a_{2s}[u_k-h(u_k)] \to 0$ because $u_k \to 1$, we have in particular $a_{2s}[u_k-h_k] \to 0$. Using this and $a_{2s}[u_k, h_k] = a_{2s}[h_k]$ by orthogonality, we write 
\[ o(1) = a_{2s}[u_k] - a_{2s}[h_k]= a_{2s}[1] - c_k^2 a_{2s}[1] + o(1) \]
and hence $c_k = 1 + o(1)$. From this we get 
\begin{align}\nonumber 
    o(1)&= a_{2s}[u_k] - a_{2s}[u_k, h_k] = a_{2s}[1] - \alpha_{2s}(0) (1 + o(1)) \int_{\mathbb S^n} v_{\zeta_k} \\&= \alpha_{2s}(0) (|\mathbb S^n| - (1 + o(1)) \int_{\mathbb S^n} v_{\zeta_k}.     \label{eq proof ift d(u)}
\end{align} 
If $|\zeta_k| \to 1$, then it is easy to see that $\int_{\mathbb S^n} v_{\zeta_k} \to \infty$ and we obtain a contradiction. Thus we may assume $\zeta_k \to \zeta$ for some $\zeta \in \R^{n+1}$ with $\delta_0 \leq |\zeta| < 1$ and $h_k \to h = v_\zeta$. But now the reverse Hölder inequality \eqref{reverse hölder} gives 
\[ \int_{\mathbb S^n} v_{\zeta_k} = (1 + o(1)) \int_{\mathbb S^n} v_{\zeta} > (1 + o(1))\|v_\zeta\|_p \|1\|_p^{p-1} = (1 + o(1))|\mathbb S^n|. \]
where the strictness comes from the equality condition in \eqref{reverse hölder}, which is not satisfied because $\zeta \neq 0$. This is a contradiction to \eqref{eq proof ift d(u)}, and the proof is therefore complete.

\section{Towards the existence of a stability minimizer}
\label{section minimizer}

We now give the proof of Theorem \ref{theorem strict inequality}. Since, with the above tools available (in particular Proposition \ref{proposition bahri-coron}), it presents no additional difficulties with respect to the standard case $s < \frac{n}{2}$ treated in \cite{Koenig2022-stab}, we will be brief.

\begin{proof}
[Proof of Theorem \ref{theorem strict inequality}]
As in the proof of Proposition \ref{proposition local 2}, we develop $\mathcal E(u_\eps)$. This time we take $u_\eps = 1 + \eps \rho$ with $\rho \in E_2$ to be chosen later, and $\eps \to 0$, and we Taylor expand to third order. We obtain 
 \begin{align*}
  \left( \int_{\mathbb S^n} (1 + \rho_k)^p \right)^{2/p}  &= \int_{\mathbb S^n} 1 + (p-1) \eps^2 \left( \int_{\mathbb S^n} 1 \right)^{\frac{2}{p} - 1} \int_{\mathbb S^n} \rho^2 \\
  &\quad + \frac{(p-1)(p-2)}{3} \eps^3 \left( \int_{\mathbb S^n} 1 \right)^{\frac{2}{p} - 1} \int_{\mathbb S^n} \rho^3  + o(\eps^3). 
\end{align*} 
Moreover, by Proposition \ref{proposition bahri-coron} we have $\mathsf d(1 + \eps \rho) = a_{2s}[\rho]$ for every $\eps > 0$ small enough.  Altogether, for every $\rho \in E_{2}$ we get 
 \[ \mathcal E(u_\eps) = \frac{4s}{n + 2s + 2} - \eps \mathcal S_s \frac{(p-1)(p-2)}{3} \frac{\int_{\mathbb S^n} \rho^3}{a_{2s}[\rho]} + o (\eps). \]
 This expansion yields $\mathcal E(u_\eps) < \frac{4s}{n+2s+2}$ for every $\eps > 0$ small enough, provided that we can choose $\rho \in E_2$ such that $\int_{\mathbb S^n} \rho^3 > 0$. As detailed in \cite{Koenig2023}, if $n \geq 2$, we can pick $\rho(\omega) = \omega_1 \omega_2 + \omega_2 \omega_3 + \omega_3 \omega_1$, which does the job. 
\end{proof}

Finally, we prove Proposition \ref{theorem minimizer}.

\begin{proof}
[Proof of Proposition \ref{theorem minimizer}]
Let $u_k > 0$ be a minimizing sequence for $\mathcal E$. Up to a conformal transformation and multiplication by a constant, which do not change the value of $\mathcal E$, we may assume that $u_k = 1 + \rho_k$, with $\rho_k \in E_{\geq 2}$ and that $\mathsf d(u_k) = a_{2s}[\rho_k]$. 

\textit{Step 1: Weak limit. } We begin by noticing that $a_{2s}[\rho_k]$ must be bounded. Indeed, we have 
\[ \mathcal E(u_k) = \frac{a_{2s}[1] + a_{2s}[\rho_k] - \mathcal S_s \|u_k\|_p^2}{a_{2s}[\rho_k]} \geq \frac{a_{2s}[1] + a_{2s}[\rho_k]}{a_{2s}[\rho_k]} \]
because $\mathcal S_s < 0$ for $s - \frac{n}{2} \in (0,1)$. Now if $a_{2s}[\rho_k] \to \infty$, we would get that 
\[ c_{BE}(s) = \lim_{k \to \infty} \mathcal E(u_k) \geq 1 \]
in contradiction to $c_{BE}(s) \leq c_{BE}^\text{loc}(s) = \frac{4s}{n+2s+2} < 1$. 

Thus $\rho_k$ is bounded in $H^s(\mathbb S^n)$ (by norm equivalence on $E_{\geq 2}$). Therefore we may assume, up to extracting a subsequence and relabeling it, that there exists $\rho \in H^s(\mathbb S^n)$ such that $\rho_k \rightharpoonup \rho$ weakly in $H^s(\mathbb S^n)$. We write 
\[ \rho_k = \rho + \eta_k \]
and notice that $1 + \rho \geq 0$, that $\rho, \eta_k \in E_{\geq 2}$ and that $\eta_k \to 0$ uniformly on $\mathbb S^n$ (by compact embedding $H^s(\mathbb S^n) \to C(\mathbb S^n)$). Moreover, by the same argument as above, $a_{2s}[\eta_k]$ is bounded.

\textit{Step 2: The weak limit is not in $\mathcal M$. } Let us show that $\rho \neq 0$. By contradiction, if $\rho = 0$, then 
\[ \mathcal E(u_k) = \frac{a_{2s}[1] + a_{2s}[\eta_k] - \mathcal S_s \|1 + \eta_k\|_p^2}{a_{2s}[\eta_k]} \]
Suppose first that (up to a subsequence) $a_{2s}[\eta_k] \to T$ for some $T > 0$. Since $\eta_k \to 0$ uniformly, we have $\mathcal S_s \|1 + \eta_k\|_p^2 \to \mathcal S_s  \|1\|_p^2 = a_{2s}[1]$ and hence 
\[c_{BE}(s) + o(1) = \mathcal E(u_k) = \frac{T + o(1)}{T} = 1 + o(1) \]
in contradiction to $c_{BE}(s) <1$, as in Step 1. 

So the only possibility left is $a_{2s}[\eta_k] \to 0$ (up to a subsequence). But then $\eta_k \to 0$ strongly in $H^s$, and our local analysis, Proposition \ref{proposition local 2}, yields 
\[ c_{BE}(s) + o(1) = \mathcal E(u_k) \geq c_{BE}^\text{loc}(s). \]
But this is in contradiction to the strict inequality $c_{BE}(s) < c_{BE}^\text{loc}(s)$ from Theorem \ref{theorem strict inequality}. 

Since the assumption $\rho =0$ yields a contradiction in all cases, we conclude that $\rho = 0$ cannot occur. This implies $1 + \rho \notin \mathcal M$. Indeed, if $1 + \rho \in \mathcal M$, then $1 + \rho(\omega) = c (1 - \zeta \cdot \omega)^{-\frac{2s-n}{2}}$ for some $c > 0$ and $|\zeta| < 1$. Since $\rho \neq 0$, we have $\zeta \neq 0$. Let $i$ be such that $\zeta_i \neq 0$. Then 
\[ 0 = \int_{\mathbb S^n} (1 + \rho) \omega_i = \int_{\mathbb S^n}  c (1 - \zeta \cdot \omega)^{-\frac{2s-n}{2}} \omega_i \neq 0, \]
and contradiction. Hence $1 + \rho \notin \mathcal M$, as claimed. 

The pointwise uniform convergence $u_k \to 1 + \rho$ follows from the compact embedding $H^s(\mathbb S^n) \to L^\infty(\mathbb S^n)$. As a consequence, $u_k > 0$ implies that $1 + \rho \geq 0$. This completes the proof of Proposition \ref{theorem minimizer}. 
\end{proof}

\appendix 

\section{Some computations}
The following lemma provides the asymptotic behavior of the eigenvalues $\alpha_{2s}(k) = \frac{\Gamma(k + \frac{n}{2} + s)}{\Gamma(k + \frac{n}{2} - s)}$ of the operator $A_{2s}$. 

\begin{lemma}
\label{lemma alpha(k) asymptotics}
As $k \to \infty$, we have $\alpha_{2s}(k) = k^{2s} (1 + \mathcal O(k^{-1}))$. 
\end{lemma}

\begin{proof}
By Stirling's formula, 
\[ \Gamma(z) = \sqrt{ \frac{2 \pi}{z}} \left( \frac{z}{e}  \right)^z ( 1+ \mathcal O(z^{-1})) \qquad \text{ as } z \to \infty. \]
With $z = k + \frac{n}{2} - s$, we thus have 
\begin{align*} \alpha_{2s}(k) &= \frac{\Gamma(z + 2s)}{\Gamma(z)} = \sqrt{\frac{z}{z+2s}} \frac{\left( \frac{z+2s}{e}  \right)^{z + 2s}}{\left( \frac{z}{e}  \right)^z} (1 + \mathcal O(z^{-1})) \\
& = e^{-2s} (z + 2s)^{2s} \left( \frac{z + 2s}{z} \right)^z (1 + \mathcal O(z^{-1})). 
\end{align*}
Now the facts that $(z + 2s)^{2s} = z^{2s} (1 + \mathcal O(z^{-1}))$ and $\left( \frac{z + 2s}{z} \right)^z = e^{2s} (1 + \mathcal O(z^{-1}))$, together with $z^{-1} = \mathcal O(k^{-1})$, give the claim. 
 \end{proof}

For the following lemma, we recall the notations $v_{\zeta} = (1 - |\zeta|^2)^\frac{n}{2p} (1 - \zeta \cdot \omega)^{-\frac{n}{p}}$ and $B_\lambda = \lambda^\frac{n}{p} B(\lambda x)$ with $B(x)  = (\frac{2}{1 + |x|^2})^{n/p}$, and $p = \frac{2n}{n-2s}$. 

\begin{lemma}
\label{lemma int u v^{p-1}}
Let $(\zeta_k) \subset \R^{n+1}$ with $|\zeta_k| < 1$ and $\zeta_k \to \nu$ for some $\nu \in \mathbb S^n$. Let $u > 0$ with $u \in H^s(\mathbb S^n)$. Then 
\[ \int_{\mathbb S^n} u v_{\zeta_k}^{p-1} =  u(\nu) (1 - |\zeta_k|)^{\frac{n}{2p}} (c_0+ o(1)) \]
for 
\[ c_0 = 2^{-\frac{n}{2p}} \int_{\R^n} B^{p-1} .  \]
\end{lemma}

\begin{proof}
Let $R_k$ be a rotation of $\R^{n+1}$ such that $R_k^{-1} \zeta_k= \beta_k e_{n+1}$, with $\beta_k := |\zeta_k|$. Using stereographic projection, a direct computation shows 
\begin{align*}
    \int_{\mathbb S^n} u v_{\zeta_k}^{p-1} = \int_{\mathbb S^n} u(R_k \omega) v_{\beta_k e_{n+1}}^{p-1} = \int_{\R^n} u(R_k \mathcal S( x)) B_{\lambda_k}^{p-1} B  = \int_{\R^n} u(R_k \mathcal S(\lambda_k^{-1} x)) B^{p-1} B_{\lambda_k^{-1}},
\end{align*}
with $\lambda_k = \left( \frac{1 +\beta_k}{1- \beta_k} \right)^{1/2}$. Since $\lambda_k \to \infty$ as $\beta \to 1$, dominated convergence yields 
\begin{align*}
     \int_{\R^n} u(R_k \mathcal S(\lambda_k^{-1} x)) B^{p-1} B_{\lambda_k^{-1}} &= \lambda_k^{-\frac{n}{p}} \left( \lim_{k \to \infty} u(R_k(\mathcal S(0)) B(0) \int_{\R^n} B^{p-1} + o(1) \right) \\
     &=  u(\nu) (1 - \beta_k)^{\frac{n}{2p}} (c_0+ o(1)), 
\end{align*}
as claimed. 
\end{proof}

%%%%%%%%%%%%%%%%%%%%%%%%%%%%%%%%%%%%%%%%%%%
%%%%%%%%%%%%%%%%%%%%%%%%%%%%%%%%%%%%%%%%%%%

\bibliography{References}
	\bibliographystyle{plain}

\end{document}